\documentclass{amsart}
\usepackage{amsfonts,amssymb,amsmath,amsthm}
\usepackage{url}
\usepackage{enumerate}
\usepackage[all]{xy}
\newtheorem{theorem}{Theorem}[section]
\newtheorem{lem}[theorem]{Lemma}
\newtheorem{prop}[theorem]{Proposition}
\newtheorem{cor}[theorem]{Corollary}
\theoremstyle{definition}
\newtheorem{definition}[theorem]{Definition}

\newtheorem{remark}[theorem]{Remark}
\newtheorem{example}[theorem]{Example}

\numberwithin{equation}{section}

\newcommand{\Ass}{\operatorname{Ass}}

\newcommand{\Spec}{\operatorname{Spec}}

\newcommand{\Ext}{\operatorname{Ext}}

\newcommand{\BN}{\Bbb N}
\newcommand{\BZ}{\Bbb Z}

\newcommand{\fm}{\frak{m}}
\newcommand{\fp}{\frak{p}}
\newcommand{\fq}{\frak{q}}

\newcommand{\init}{\operatorname{in}}

\begin{document}

\title[CCM property and Lyubeznik numbers under Gr\"obner deformations]{Canonical Cohen-Macaulay property and Lyubeznik numbers under Gr\"obner deformations}
\author{Parvaneh Nadi}
\email{nadi$_{-}$p@aut.ac.ir}
\address{Department of Mathematics and Computer Science, Amirkabir University of Technology, 424 Hafez Av, Tehran, 1591634311, Iran.}
\author{Matteo Varbaro} 
\email{varbaro@dima.unige.it}
\address{Dipartimento di Matematica, Universit\`a di Genova, Italy} 
  \date{}
\maketitle
\begin{abstract}
In this note we draw some interesting consequences of the recent results on squarefree Gr\"obner degenerations obtained by Conca and the second author.
\end{abstract}
\section{Introduction}

Let $R=K[x_1,...,x_n]$ be a positively graded polynomial ring over a field $K$, where $x_i$ is homogeneous of degree $g_i\in\mathbb{N}_{>0}$, and $\mathfrak{m}=(x_1,\ldots ,x_n)$ denotes its homogeneous maximal ideal. Also denote the canonical module of $R$ by $\omega_{R} =R(-|g|)$, where $|g|=g_1+\ldots +g_n$.

\begin{definition}
A graded finitely generated $R$-module $M$ is called canonical  Cohen-Macaulay (CCM for short) if $\Ext^{n-\dim M}_R(M,\omega_R)$ is Cohen-Macaulay.
\end{definition}

This notion was introduced by Schenzel in \cite{S}, who proved in the same paper the following result that contributes to make it interesting: given a homogeneous prime ideal $I\subset R$, the ring $R/I$ is CCM if and only if it admits a birational Macaulayfication (that is a birational extension $R/I\subset A\subset Q(R/I)$ such that $A$ is a finitely generated Cohen-Macaulay $R/I$-module, where $Q(R/I)$ is the fraction field of $R/I$). In this case, furthermore, $A$ is the endomorphism ring of $\Ext^{n-\dim R/I}_R(R/I,\omega_R)$.

In this note, we will derive by the recent result obtained by Conca and the second author in \cite{C-V} the following: if a homogeneous ideal $I\subset R$ has a radical initial ideal $\init_{\prec}(I)$ for some monomial order $\prec$, then $R/I$ is CCM whenever $R/\init_{\prec}(I)$ is CCM. In fact we prove something more general, from which we can also infer that, in positive characteristic, under the same assumptions the Lyubeznik numbers of $R/I$ are bounded above from those of $R/\init_{\prec}(I)$. As a consequence of the latter result, we can infer that, also in characteristic 0 by reduction to positive characteristic, if $\init_{\prec}(I)$ is a radical monomial ideal the following are equivalent:

\begin{enumerate}
\item The  dual graph (a.k.a. Hochster-Huneke graph) of $R/I$ is connected.
\item The dual graph of $R/\init_{\prec}(I)$ is connected.
\end{enumerate}

Motivated by these results, in the last section we study the CCM property for Stanley-Reisner rings $K[\Delta]$. We show that $K[\Delta]$ is CCM whenever $\Delta$ is a simply connected 2-dimensional simplicial complex.

\section{CCM, Lyubeznik numbers and Gr\"obner deformations}

Throughout this section, let us fix a monomial order $\prec$ on $R$. We start with the following crucial lemma:

\begin{lem}\label{mainlemma}
Let $I$ be a homogeneous ideal of $R$ such that $\mathrm{in}_{\prec}(I)$ is radical. Then, for all $i,j,k\in\BZ$, we have:

\begin{center}
$\dim_K\Ext^i_R(\Ext_R^j(R/I,\omega_R),\omega_R)_k\leq \dim_K\Ext^i_R(\Ext_R^j(R/\init_{\prec}(I),\omega_R),\omega_R)_k$
\end{center}

\end{lem}

\begin{proof}
Let $w=(w_{1},...,w_{n}) \in \mathbb{N}^{n}$ be a weight such that $\mathrm{in}_{w}(I)=\mathrm{in}_{\prec}(I)$. Let $t$ be a new indeterminate over $R$. Set $P=R[t]$ and $S=P/\hom_{w}(I)$.  By providing $P$ with the graded structure given by $\deg(x_i)=g_i$ and $\deg(t)=0$, $\hom_w(I)$ is homogeneous. If $x\in \{t,t-1\}$, apply the functor $\Ext_P^i(\Ext_P^j(S,P),-)$ to the short exact sequence
\[
0 \rightarrow P \xrightarrow{\cdot x} P \rightarrow P/xP \rightarrow 0 
\]
getting the short exact sequences
\begin{center}
$0\rightarrow \mathrm{Coker} \mu_{x}^{i-1,j} \rightarrow \mathrm{Ext}^{i}_{P}(\mathrm{Ext}^{j}_{P}(S,P),P/xP) \rightarrow \mathrm{Ker} \mu_{x}^{i,j} \rightarrow 0.$
\end{center}
where $\mu_x^{i,j}$ is the multiplication by $x$ on $\Ext_P^i(\Ext_P^j(S,P),P)$. So, for all $k\in \BZ$ we have exact sequences of $K$-vector spaces: 
\begin{center}
$0\rightarrow [\mathrm{Coker} \mu_{x}^{i-1,j}]_k \rightarrow \mathrm{Ext}^{i}_{P}(\mathrm{Ext}^{j}_{P}(S,P),P/xP)_k \rightarrow [\mathrm{Ker} \mu_{x}^{i,j}]_k \rightarrow 0.$
\end{center}
Since $E^{i,j}_k=\Ext_P^i(\Ext_P^j(S,P),P)_k$ is a finitely generated graded (w.r.t. the standard grading) $K[t]$-module, we can write $E^{i,j}_k=F^{i,j}_k+T^{i,j}_k$ where $F^{i,j}_k=K[t]^{f^{i,j}_k}$ and $T^{i,j}_k=\bigoplus_ {r=1}^{g^{i,j}_k}K[t]/(t^{d_r})$ with $d_r\geq 1$. Therefore we have:
\[\dim_K[\mathrm{Coker} \mu_{t-1}^{i-1,j}]_k=f^{i-1,j}_k\leq f^{i-1,j}_k+g^{i-1,j}_k=\dim_K[\mathrm{Coker} \mu_{t}^{i-1,j}]_k\] 
and 
\[\dim_K[\mathrm{Ker} \mu_{t-1}^{i,j}]_k=0\leq g^{i,j}_k=\dim_K[\mathrm{Ker} \mu_{t}^{i-1,j}]_k.\]
So $\dim_K \mathrm{Ext}^{i}_{P}(\mathrm{Ext}^{j}_{P}(S,P),P/(t-1)P)_k\leq \dim_K\mathrm{Ext}^{i}_{P}(\mathrm{Ext}^{j}_{P}(S,P),P/tP)_k$. 

Note that, by using \cite[Proposition 1.1.5]{B-H} one can infer the following: if $A$ is a ring, $M$ and $N$ are $A$-modules, and $a\in \mathrm{Ann}(N)$ is $A$-regular and $M$-regular, then
\[\Ext^i_A(M,N)\cong \Ext^i_{A/aA}(M/aM,N) \ \ \forall \ i\in \BN.\] 
Since by \cite[Proposition 2.4]{C-V} $\mathrm{Ext}^{j}_{P}(S,P)$ is a flat $K[t]$-module, the multiplication by $x$ on it is injective: that is, $x$ is $\mathrm{Ext}^{j}_{P}(S,P)$-regular. Therefore we have:
\[\mathrm{Ext}^{i}_{P}(\mathrm{Ext}^{j}_{P}(S,P),P/xP)\cong \mathrm{Ext}^{i}_{P/xP}(\mathrm{Ext}^{j}_{P}(S,P)/x\mathrm{Ext}^{j}_{P}(S,P),P/xP).\]
Again because the multiplication by $x$ is injective on $\mathrm{Ext}^{j}_{P}(S,P)$ and by the property mentioned above, we have
\[\mathrm{Ext}^{j}_{P}(S,P)/x\mathrm{Ext}^{j}_{P}(S,P)\cong \mathrm{Ext}^{j}_{P}(S,P/xP)\cong \mathrm{Ext}^{j}_{P/xP}(S/xS,P/xP) .\]
Putting all together we get:
\begin{eqnarray*}
\dim_K \mathrm{Ext}^{i}_{P/(t-1)P}(\mathrm{Ext}^{j}_{P/(t-1)P}(S/(t-1)S,P/(t-1)P),P/(t-1)P)_k\leq \\
 \dim_K\mathrm{Ext}^{i}_{P/tP}(\mathrm{Ext}^{j}_{P/tP}(S/tS,P/tP),P/tP)_k,
 \end{eqnarray*}
 that, because $\omega_R\cong R(-|g|)$, is what we wanted:
 \[\dim_K\mathrm{Ext}^{i}_{R}(\mathrm{Ext}^{j}_{R}(R/I,R),R)_k\leq \dim_K\mathrm{Ext}^{i}_{R}(\mathrm{Ext}^{j}_{R}(R/\init_{\prec}(I),R),R)_k.\]
\end{proof}

\begin{cor}\label{c:CCM}
Let $I$ be a homogeneous ideal of $R$ such that $\mathrm{in}_{\prec}(I)$ is radical. Then, $R/I$ is canonical Cohen-Macaulay whenever $R/\init_{\prec}(I)$ is so.
\end{cor}
\begin{proof}
For a homogeneous ideal $J\subset R$, $R/J$ is CCM if and only if 
\[\Ext^{n-i}_R(\Ext_R^{n-\dim R/J}(R/J,\omega_R),\omega_R)=0 \ \ \forall \ i<\dim R/J,\] so the result follows from Lemma \ref{mainlemma}.
\end{proof}

\begin{remark}\label{r:gin}
Corollary \ref{c:CCM} fails without assuming that $\init_{\prec}(I)$ is radical. In fact, if $\prec$ is a degrevlex monomial order and $I$ is in generic coordinates, by \cite[Theorem 2.2]{H-S} $R/\init_{\prec}(I)$ is sequentially Cohen-Macaulay, thus CCM (for example see \cite[Theorem 1.4]{H-S}). However, it is plenty of homogeneous ideals $I$ such that $R/I$ is not CCM.
\end{remark}

We do not know whether the implication of Corollary \ref{c:CCM} can be reversed. 
Without assuming that $\init_{\prec}(I)$ is radical, we already noticed that Corollary \ref{c:CCM} fails in Remark \ref{r:gin}. The following example shows that in general $R/I$ CCM but $R/\init_{\prec}(I)$ not CCM can also happen:

\begin{example}
Let $R=K[x_1,...,x_9]$ and 
\[I=(x_1^3+x_2^3,x_5^2x_9+x_4^2x_8,x_5^3x_7+x_6^3x_9,x_7^2x_1+
x_6^2x_5,x_3x_9-x_4x_8).\]
Since $I$ is a complete intersection, $R/I$ is CCM. However one can check that, if $\prec$ is the lexicographic order extending $x_1>\ldots >x_9$, $R/\init_{\prec}(I)$ is not CCM.
\end{example}

\subsection{Lyubeznik numbers and connectedness}

Let $I\subset R=K[x_1,\ldots ,x_n]$. In \cite{Lyu} Lyubeznik introduced the following invariants of $A=R/I$:

\[\lambda_{i,j}(A) =
\mathrm{dim}_{K}\mathrm{Ext}^{i}_{R}(K,H_{I}^{n-j}(R)) \ \ \forall \ i,j\in \BN .\]

It turns out that these numbers, later named {\it Lyubeznik numbers}, depend only on $A$, $i$ and $j$, in the sense that if $A\cong S/J$ where $J\subset S=K[y_1,\ldots ,y_m]$,
\[\lambda_{i,j}(A) =
\mathrm{dim}_{K}\mathrm{Ext}^{i}_{S}(K,H_{J}^{m-j}(S)) \ \ \forall \ i,j\in \BN .\]
Also, $\lambda_{i,j}(A)=0$ whenever $i>j$ or $j>\dim A$, and $\lambda_{d,d}(A)$ is the number of connected components of the {\it dual graph} (also known as the {\it Hochster-Huneke graph}) of $A\otimes_K\overline{K}$, \cite{Zh}. (We recall that the dual graph of a Noetherian ring  $A$ of dimension $d$ is the graph with the minimal primes of $A$ as vertices and such that $\{\fp,\fq\}$ is an edge if and only if $\dim A/(\fp+\fq)=d-1$). We will refer to the upper triangular matrix $\Lambda(A)=(\lambda_{i,j}(A))$ of size $(\dim A+1)\times (\dim A+1)$ as the {\it Lyubeznik table} of $A$. By {\it trivial Lyubeznik table} we mean that $\lambda_{\dim A,\dim A}(A)=1$ and $\lambda_{i,j}(A)=0$ otherwise.

\begin{cor}\label{c:lyu}
Let $I$ be a homogeneous ideal of $R$ such that $\mathrm{in}_{\prec}(I)$ is radical. If $K$ has positive characteristic,
\[\lambda_{i,j}(R/I)\leq \lambda_{i,j}(R/\init_{\prec}(I)) \ \ \forall \ i,j\in\BN.\]
\end{cor}
\begin{proof}
By \cite[Theorem 1.2]{Z}, if $J\subset R$ is a homogeneous ideal,
\[\lambda_{i,j}(R/J)=\mathrm{dim}_{K}(\mathrm{Ext}^{n-i}_{R}(\mathrm{Ext}^{n-j}_{R}(R/J,\omega_{R}),\omega_{R})_{0})_{s},\]
where the subscript $(-)_s$ stands for the stable part under the natural Frobenius action. In particular
\[\lambda_{i,j}(R/J)\leq\mathrm{dim}_{K}\mathrm{Ext}^{n-i}_{R}(\mathrm{Ext}^{n-j}_{R}(R/J,\omega_{R}),\omega_{R})_{0}.\]
On the other hand, if $J\subset R$ is a radical monomial ideal, Yanagawa proved in \cite[Corollary 3.10]{Y} (independently of the characteristic of $K$) that:
\[\lambda_{i,j}(R/J)=\mathrm{dim}_{K}\mathrm{Ext}^{n-i}_{R}(\mathrm{Ext}^{n-j}_{R}(R/J,\omega_{R}),\omega_{R})_{0}.\]
So the result follows from Lemma \ref{mainlemma}.
\end{proof}
The following two examples show that Corolarry \ref{c:lyu} is false without assuming both that $\mathrm{in}_{\prec}(I)$ is radical and that $K$ has positive characteristic:

\begin{example}\cite[Example 4.11]{D-G-N}
Let $R=K[x_1,...,x_6]$ and $\mathrm{char}(K)=5$. Let 
\begin{eqnarray*}
I=(x_4^3+x_5^{3}+x_6^{3},x_4^{2}x_1+x_5^{2}x_2+x_6^{2}x_3,x_1^{2}x_4+x_2^{2}x_5+x_3^{2}x_6, \\
x_1^{3}+x_2^{3}+x_3^{3},x_5x_3-x_6x_2,x_6x_1-x_4x_3,x_4x_2-x_5x_1).
\end{eqnarray*}
 
Then
\[
\Lambda(R/I)=\begin{bmatrix}
 0&0&1 &0 \\
 & 0& 0 &0\\
 & &0&1\\
&& &1\\
\end{bmatrix}.
\]

If $\prec$ is the degree reverse lexicographic term order extending $x_1>\ldots >x_6$ one has:
\[\mathrm{in}_{\prec}(I)=(x_3x_5,x_3x_4,x_2x_4,x_4^{3},x_1x_4^{2},x_1^{2}x_4,x_1^{3}).\]
One can check that $R/\mathrm{in}_{\prec}(I)$ has a trivial Lyubeznik table.
\end{example}

\begin{example}
Let $K$ be a field of characteristic $0$ and $R=K[x_1,...,x_6]$. Let $I=(x_1x_5-x_2x_4,x_1x_6-x_3x_4,x_2x_6-x_3x_5)$. By \cite[Example 2.2]{Al-Y},
Lyubeznik table of $R/I$ is 
\[
\Lambda(R/I)=\begin{bmatrix}
 0&0&0&1 &0 \\
& 0& 0& 0 &0\\
& & 0&0&1\\
&&& 0&0\\
&&&&1\\
\end{bmatrix}.
\]
If $\prec$ is the degree reverse lexicographic term order extending $x_1>\ldots >x_6$ we have
\[\mathrm{in}_{\prec}(I)=(x_2x_4,x_3x_4,x_3x_5).\] 
So $\init_{\prec}(I)$ is a radical monomial ideal, however $\Lambda(R/\init_{\prec}(I))$ is trivial.
\end{example}

In Corollary \ref{c:lyu} we have an equality when $R/I$ is generalized Cohen-Macaulay:

\begin{cor}
Let $I$ be a homogeneous ideal of $R$ such that $\mathrm{in}_{\prec}(I)$ is radical. If $K$ has positive characteristic and $R/I$ is generalized Cohen-Macaulay,
\[\lambda_{i,j}(R/I)= \lambda_{i,j}(R/\init_{\prec}(I)) \ \ \forall \ i,j\in\BN.\]
\end{cor}
\begin{proof}
Since $R/I$ is generalized Cohen-Macaulay so is  $R/\mathrm{in}_{\prec}(I)$ by \cite[Corollary 2.11]{C-V}. Therefore it is enough to show that $\lambda_{0,j}(R/I)=\lambda_{0,j}(R/\mathrm{in}_{\prec}(I))$ for all $j$ (see \cite[Subsection 4.3]{Al-Y}). By \cite[Proposition 3.3]{C-V}, both $R/\init_{\prec}(I)$ and $R/I$ are cohomologically full. So from \cite[Proposition 4.11]{D-S-M}:
\begin{eqnarray*}
\lambda_{0,j}(R/I)=\mathrm{dim}_{K}[H^{j}_{\fm}(R/I)]_{0} , \\
\lambda_{0,j}(R/\init_{\prec}(I))=\mathrm{dim}_{K}[H^{j}_{\fm}(R/\init_{\prec}(I))]_{0}.
\end{eqnarray*}
Now by \cite[Theorem 1.3]{C-V} we get the result.

%
\end{proof}

\begin{remark}\label{0top}
Let $I$ be an ideal of $R=K[x_1,\ldots ,x_n]$ such that $\init_{\prec}(I)$ is generated by monomials $u_1,\ldots ,u_r$. Suppose that $K$ has characteristic 0. Since $I$ is finitely generated, there exists a finitely generated $\BZ$-algebra $A\subset K$ such that $I$ is defined over $A$, i.e. $I'R=I$ if $I'=I\cap A[x_1,\ldots ,x_n]$. Given a prime number $p$ and a prime ideal $\fp\in\mathrm{Spec} A$ minimal over $(p)$, let $Q(\fp)$ denote the field of fractions of $A/\fp$ (note that $Q(\fp)$ has characteristic $p$), $R(\fp)=Q(\fp)[x_1,\ldots ,x_n]$ and $I(\fp)=I'R(\fp)$. We call the objects $R(\fp),I(\fp),R(\fp)/I(\fp)$ reductions mod $p$ of $R,I,R/I$, and by abusing notation we denote them by $R_p,I_p,R_p/I_p$.

Seccia proved in \cite{Se} that 
\[\init_{\prec}(I_p)=\init_{\prec}(I)_p\]
for any reduction mod $p$ if $p$ is a large enough prime number, 
i.e. $\init_{\prec}(I_p)$ is generated by $u_1,\ldots ,u_r$. 
\end{remark}

\begin{remark}
Let $A$ be a Noetherian ring of dimension $d$. The ring $A$ is said to be connected in codimension 1 if $\Spec A\setminus V(\mathfrak{a})$ is connected whenever $\dim A/\mathfrak{a}<d-1$ (here $V(\mathfrak{a})$ denotes the set of primes containing $\mathfrak{a}$). A result of Hartshorne \cite[Proposition 1.1]{Ha} implies that the dual graph of $A$ is connected if and only if $A$ is connected in codimension 1.
\end{remark}

\begin{prop}
Let $I$ be a homogeneous ideal of $R$ such that $\mathrm{in}_{\prec}(I)$ is radical. Then:
\begin{enumerate}
\item $\mathrm{Proj} R/I$ is connected if and only if $\mathrm{Proj} R/\init_{\prec}(I)$ is connected.
\item The  dual graph of $R/I$ is connected if and only if the dual graph of $R/\init_{\prec}(I)$ is connected.
\end{enumerate}
\end{prop}
\begin{proof}
The ``only if" parts hold without the assumption that $\init_{\prec}(I)$ is radical and they have been proved in \cite{Va}. So we will concentrate on the ``if" parts.

Since computing initial ideal, as well as the connectedness properties concerning $R/\init_{\prec}(I)$, are not affected extending the field, while the connectedness properties concerning $R/I$ follow from the corresponding connectedness properties of $R/I\otimes_K\overline{K}$, it is harmless to assume that $K$ is algebraically closed. 
Under this assumption, if $J\subset R$ is a homogeneous radical ideal, we have that:
\begin{itemize}
\item[(a)] $\mathrm{Proj} R/J$ is connected if and only if $H^1_{\fm}(R/J)_0=0$.
\item[(b)] The dual graph of $R/J$ is connected if and only if $\lambda_{\dim R/J,\dim R/J}(R/J)=1$ by the main theorem of \cite{Zh}.
\end{itemize}
Under our hypothesis $I$ is radical, so (1) follows at once from (a) and the fact that the Hilbert function of the local cohomology modules of $R/I$ is bounded above by that of the ones of $R/\init_{\prec}(I)$ (in this case we even have equality by \cite{C-V}). Concerning the ``if-part'' of (2), since $\lambda_{\dim R/I,\dim R/I}(R/I)\neq 0$ in any case, if $K$ has positive characteristic it follows from (b) and Corollary \ref{c:lyu}. So, assume that $K$ has characteristic 0. If, by contradiction, $R/I$ were not connected in codimension 1, there would be two ideals $H\supsetneq I$ and $J\supsetneq I$ such that $H\cap J=I$ and $\dim R/(H+J)<\dim R/I-1$ (see \cite[Lemma 19.1.15]{B-S}). By Remark \ref{0top}, it is not difficult to check that we can choose a prime number $p\gg 0$ such that $H_p\supsetneq I_p$ and $J_p\supsetneq I_p$, $H_p\cap J_p=I_p$, $\dim R_p/(H_p+J_p)<\dim R_p/I_p-1$ and $\init_{\prec}(I_p)=\init_{\prec}(I)_p$ (for instance, to compute the intersection of two ideals amounts to perform a Gr\"obner basis calculation). Clearly the dual graph of a Stanley-Reisner ring does not depend on the characteristic of the base field. So the dual graph of $R_p/\init_{\prec}(I_p)$ would be connected but that of $R_p/I_p$ would be not, and this contradicts the fact that we already proved the result in positive characteristic.
\end{proof}

\section{CCM Simplicial Complexes}

Let $\Delta$ be a simplicial complex on the vertex set $[n]=\{1,...,n\}$. We denote the Stanley-Reisner ring $R/I_{\Delta}$ by $K[\Delta]$. See \cite{St} for generalities on these objects.
The aim of this section is to examine the CCM property for the Stanley-Reisner rings $K[\Delta]$, especially when $\Delta$ has dimension 2. 

Recall that a $\BN^n$-graded $R$-module $M$ is {\it squarefree} if, for all $\alpha=(\alpha_1,\ldots ,\alpha_n)\in\BN^n$, the multiplication by $x_j$ from $M_{\alpha}$ to $M_{\alpha+e_j}$ is bijective whenever $\alpha_j\neq 0$. It turns out that $K[\Delta]$, $I_{\Delta}$ and $\Ext^i_R(K[\Delta],\omega_R)$ are squarefree modules by \cite{Yan}.

\begin{lem}\label{2}
Let $M$ be a nonzero squarefree module. If $M_{0}=0$, then $\mathrm{depth}M>0$.
\end{lem}

\begin{proof}
Assume, by way of contradiction, that $\mathrm{depth} M=0$. Then $\fm\in \mathrm{Ass}M$. So there exist $\alpha=(\alpha_1,\ldots ,\alpha_n) \in \mathbb{N}^{n}$ and $0\neq u \in M_{\alpha}$ such that $\fm=\mathrm{Ann}(u)$. So for $j=1,...,n$, $x_{j}\cdot u=0$. It follows that the multiplication map on $M_{\alpha}$ by $x_{j}$ is not injective for all $j$. So, because $M$ is a squarefree module, $\alpha=0$ and $u \in M_{0}=0$, a contradiction. Hence $\mathrm{depth} M>0$.
\end{proof}

\begin{lem}\label{1}
For any homogeneous ideal $I\subset R$, for all $i<3$ the $R$-module $\mathrm{Ext}^{n-i}_{R}(\mathrm{Ext}^{n-\dim R/I}_{R}(R/I,R),R)$ has finite length.
\end{lem}
\begin{proof}
If $\left(\cap_{i=1}^r\fq_i\right)\cap \left(\cap_{j=1}^s\fq_j'\right)$ is an irredundant primary decomposition of $I$ with $\dim R/\fq_i=\dim R/I$ and $\dim R/\fq'_j>\dim R/I$, one has 
\[\mathrm{Ext}^{n-\dim R/I}_{R}(R/I,R)\cong \mathrm{Ext}^{n-\dim R/I}_{R}(R/\cap_{i=1}^r\fq_i,R).\]
 So we can assume that $\dim R/\fp=\dim R/I$ for all $\fp\in\Ass R/I$.

Let $\fp\neq \fm$ be a homogeneous prime ideal of $R$ containing $I$, and set $M_{i}=\mathrm{Ext}^{n-i}_{R}(\mathrm{Ext}^{n-\dim R/I}_{R}(R/I,R),R)$.
We have: 
\[(M_{i})_{\fp}=\mathrm{Ext}^{\mathrm{ht}(\fp)-(i-n+\mathrm{ht}(\fp))}_{R_{\fp}}(\mathrm{Ext}^{\mathrm{ht}(\fp)-(\dim R_{\fp}/IR_{\fp})}_{R_{\fp}}(R_{\fp}/IR_{\fp},R_{\fp}),R_{\fp}).\]
Since $i-n+\mathrm{ht}(\fp)\leq 1$ by the assumptions and $\mathrm{Ext}^{\mathrm{ht}(\fp)-(\dim R_{\fp}/IR_{\fp})}_{R_{\fp}}(R_{\fp}/IR_{\fp},R_{\fp}),R_{\fp})$ has depth at least 2 by \cite[Proposition 2.3]{S} we have $(M_{i})_{\fp}=0$.

%
\end{proof}

\begin{cor}\label{3}
Let $\Delta$ be a $2$-dimensional simplicial complex. Then $K[\Delta]$ is CCM if and only if $\lambda_{2,3}(K[\Delta])=0$.
\end{cor}

\begin{proof}
Since $\mathrm{Ext}^{n-3}_{R}(K[\Delta],\omega_{R})$ satisfy Serre's condition $(S_{2})$ by \cite[Proposition 2.3]{S}, it is enough to show that $\mathrm{Ext}^{n-2}_{R}(\mathrm{Ext}^{n-3}_{R}(K[\Delta],\omega_{R}),\omega_{R})=0$. By Lemma \ref{1} $\mathrm{Ext}^{n-2}_{R}(\mathrm{Ext}^{n-3}_{R}(K[\Delta],\omega_{R}),\omega_{R})$ has finite length; so, since it is a squarefree module,
 \[\mathrm{Ext}^{n-2}_{R}(\mathrm{Ext}^{n-3}_{R}(K[\Delta],\omega_{R}),\omega_{R})=0\iff \mathrm{Ext}^{n-2}_{R}(\mathrm{Ext}^{n-3}_{R}(K[\Delta],\omega_{R}),\omega_{R})_0=0.\]
 We conclude because $\lambda_{2,3}(K[\Delta])=\mathrm{Ext}^{n-2}_{R}(\mathrm{Ext}^{n-3}_{R}(K[\Delta],\omega_{R}),\omega_{R})_0$ by \cite[Corollary 3.10]{Y}.
%
\end{proof}

\begin{remark}
If $\Delta$ is a $(d-1)$-dimensional simplicial complex, it is still true that if $K[\Delta]$ is CCM, then $\lambda_{j,d}(K[\Delta])=0$ for all $j< d$. The converse, however, is not true as soon as $\dim(\Delta) >2$:

Let  $R=K[x_1,...,x_6]$ and $I$ be the monomial ideal of $R$ generated by
\[x_1x_2x_3x_4, x_1x_3x_4x_5, x_1x_2x_3x_6, x_1x_2x_5x_6, x_1x_4x_5x_6 ~ \text{and}~ x_3x_4x_5x_6.\]
The ring $R/I$ has a trivial Lyubeznik table but it is not CCM. Here $I$ is the Stanley-Reisner ring of a 3-dimensional simplicial complex.
\end{remark}

\begin{prop}\label{p:simply}
Let $\Delta$ be a $2$-dimensional simplicial complex such that $H_1(\Delta;K)$ vanishes. Then  $K[\Delta]$ is CCM.
\end{prop}

\begin{proof}
Since $H_1(\Delta;K)=0$, by Hochster formula we get $\mathrm{Ext}^{n-2}_{R}(K[\Delta],\omega_{R})_{0}=0$. If $ \mathrm{Ext}^{n-2}_{R}(K[\Delta],\omega_{R})\neq 0$, since it is a squarefree module it has positive depth by Lemma \ref{2}.

So, in any case, $\mathrm{Ext}^{n}_{R}(\mathrm{Ext}^{n-2}_{R}(K[\Delta],\omega_{R}),\omega_{R})=0$, and hence 
 \[\lambda_{0,2}(K[\Delta])=\mathrm{Ext}^{n}_{R}(\mathrm{Ext}^{n-2}_{R}(K[\Delta],\omega_{R})),\omega_{R})_{0}=0.\]
 By \cite[Remark 2.3]{Al-Y}, $\lambda_{2,3}(K[\Delta])=\lambda_{0,2}(K[\Delta])=0$. Now by Corollary \ref{3} $K[\Delta]$ is CCM.
\end{proof}

The converse of this corollary does not hold in general:

\begin{example}\label{e:simply}
Let $\Delta$ be the simplicial complex on $6$ vertices with facets $\{1,2,3\}$, $\{1,4,5\}$ and $\{3,4,6\}$. Then $K[\Delta]$ is CCM but $H_1(\Delta;K)\neq0$
\end{example}

\begin{prop}\label{p:B}
Let $\Delta$ be a $(d-1)$-dimensional Buchsbaum simplicial complex. The ring $K[\Delta]$ is CCM if and only if $H_i(\Delta;K)=0$ for all $1 \leq i < d-1$.
\end{prop}

\begin{proof}
 Let $K[\Delta]$ be CCM and fix $i\in \{1,\ldots ,d-2\}$. Since  $\Delta$ is Buchsbaum, $K[\Delta]$ behaves cohomologically like an isolated singularity, hence:
\[\lambda_{0,i+1}(K[\Delta])=\lambda_{d-i,d}(K[\Delta])\]
(see \cite[Subsection 4.3]{Al-Y}).
On the other hand, since the canonical module of $K[\Delta]$ is a $d$-dimensional Cohen-Macaulay module, $\lambda_{d-i,d}(K[\Delta])=0$ by \cite[Corollary 3.10]{Y}. So
\[\lambda_{0,i+1}(K[\Delta])=\mathrm{dim}_{K}\mathrm{Ext}^{n}_{R}(\mathrm{Ext}^{n-i-1}_{R}(K[\Delta],\omega_{R}),\omega_{R}) _{0}=0.\]
By local duality
 $H_{\fm}^{0}(\mathrm{Ext}^{n-i-1}_{R}(K[\Delta],\omega_{R}))_{0}=0$. Since $\mathrm{Ext}^{n-i-1}_{R}(K[\Delta],\omega_{R})$ is of finite length
\[
H_{\fm}^{0}(\mathrm{Ext}^{n-i-1}_{R}(K[\Delta],\omega_{R}))_{0}=\mathrm{Ext}^{n-i-1}_{R}(K[\Delta],\omega_{R})_{0}=0.\]
Therefore Hochster formula tells us that $H_i(\Delta;K)=0$.

Conversely, assume that $H_i(\Delta;K)=0$ for all $1 \leq i < d-1$. Then we have that $\mathrm{Ext}^{n-i-1}_{R}(K[\Delta],\omega_{R})_{0}=0$ by Hochster formula. As $\Delta$ is Buchsbaum, 
$\mathrm{Ext}^{n-i-1}_{R}(K[\Delta],\omega_{R})$ is of finite length, so 
\[
\mathrm{Ext}^{n-i-1}_{R}(K[\Delta],\omega_{R})=\mathrm{Ext}^{n-i-1}_{R}(K[\Delta],\omega_{R})_0=0 \ \ \forall \
1\leq i <d-1.
\] 
Now \cite[Theorem 4.9]{S-W} and local duality follow that for $1\leq i <d-1$,
\[
H_{\fm}^{i+1}(\mathrm{Ext}^{n-d}_{R}(K[\Delta],\omega_{R})\cong \mathrm{Ext}^{n-d+i}_{R}(K[\Delta],\omega_{R})=0.\]
Thus $K[\Delta]$ is CCM.
\end{proof}

\begin{example}
Propositions \ref{p:simply} and \ref{p:B} provide the following situation concerning CCM 2-dimensional simplicial complexes:
\begin{itemize}
\item[(i)] $H_1(\Delta;K)=0$ $\implies$ $K[\Delta]$ is CCM.
\item[(ii)] If $\Delta$ is Buchsbaum, $H_1(\Delta;K)=0$ $\iff$ $K[\Delta]$ is CCM.
\end{itemize}
Item (ii) above yields many examples of Buchsbaum 2-dimensional nonCCM simplicial complexes. We conclude this note with an example of a 2-dimensional simplicial complex which is neither Buchsbaum nor CCM: 

Let $R=K[x_1,...,x_8]$ and $\Delta$ be the simplicial complex with facets $\{x_1,x_2,x_6\}$, $\{x_2,x_6,x_4\}$, $\{x_2,x_4,x_5\}$, $\{x_2,x_3,x_5\}$, $\{x_3,x_5,x_6\}$, $\{x_1,x_3,x_6\}$, $\{x_1,x_7,x_8\}$. One can check that $\Delta$ is not Buchsbaum and $K[\Delta]$ is not CCM. Accordingly with Proposition \ref{p:simply},  $H_1(\Delta;K)\neq 0$.
\end{example}

\proof[\bf Acknowledgment]

This work was completed when the first author was visiting Department of
Mathematics of University of Genova. She wants to express her gratitude for the received hospitality.

\end{document}